\documentclass[11pt]{article}
\usepackage[margin=1in]{geometry}

\usepackage{graphicx} 
\usepackage{amsmath} 
\usepackage{amssymb}  
\usepackage{amsthm}

\usepackage{color}
\allowdisplaybreaks[1]

\usepackage{algorithm}
\usepackage{algpseudocode}
\usepackage{hyperref} 

\theoremstyle{plain}  
\newtheorem{theorem}{Theorem}[section]

\newtheorem{defn}[theorem]{Definition}
\newtheorem{lemma}[theorem]{Lemma}
\newtheorem{corollary}[theorem]{Corollary}
\newtheorem{proposition}[theorem]{Proposition}

\theoremstyle{definition}

\theoremstyle{remark} 

\newtheorem{remark}{Remark}[section]

\def\0{{\bf 0}}
\def\1{{\bf 1}}

\def \bmat{\left[\begin{matrix}}
	\def \emat{\end{matrix}\right]}

\def \xy1vec{\left[\begin{matrix}x\\y\\1\end{matrix}\right]}
\def \QED{\begin{flushright}\Halmos\end{flushright}\end{proof}}

\def\gh#1{{\color{black}#1}} 

\long\def\old#1{}

\definecolor{DarkerGreen}{RGB}{0,170,0}

\definecolor{orange}{rgb}{1,0.5,0}

\title{\LARGE \bf On the Complexity of Detecting Convexity over a Box}
\author{Amir Ali Ahmadi and Georgina Hall \thanks{The authors are with  the department of Operations Research and Financial Engineering at Princeton University. Email: \{\texttt{a\_a\_a}, \texttt{gh4}\}\texttt{@princeton.edu}. 
 This work was partially supported by the DARPA Young Faculty Award, the CAREER Award of the NSF, the Google Faculty Award, the Innovation Award of the School of Engineering and Applied Sciences at Princeton University, and the Sloan Fellowship.}}
\begin{document}
\date{}
\maketitle


\begin{abstract}
%
It has recently been shown that the problem of testing global convexity of polynomials of degree four is {strongly} NP-hard, answering an open question of N.Z. Shor. This result is minimal in the degree of the polynomial when global convexity is of concern. In a number of applications however, one is interested in testing convexity only over a compact region, most commonly a box (i.e.,  \gh{a} hyper-rectangle). In this paper, we show that this problem is also strongly NP-hard, in fact for polynomials of degree as low as three. This result is minimal in the degree of the polynomial and in some sense justifies why convexity detection in nonlinear optimization solvers is limited to quadratic functions or functions with special structure. As a byproduct, our proof shows that the problem of testing whether all matrices in an interval family are positive semidefinite is strongly NP-hard. This problem, which was previously shown to be (weakly) NP-hard by Nemirovski, is of independent interest in the theory of robust control.







%
\end{abstract}



\section{Introduction}
In a relatively recent paper~\cite{NPhard_Convexity_MathProg}, Ahmadi, Olshevsky, Parrilo, and Tsitsiklis have shown that the problem of testing whether a quartic (multivariate) polynomial is globally convex is NP-hard. This result answers a question of N.Z. Shor that had appeared in 1996 on a list of seven open problems in computational complexity for numerical optimization \cite{pardalos1992open}. The significance of the question stems from the fact that in the theory of optimization, the divide between convex and nonconvex optimization has proved to be a fundamental one. Indeed, for many classes of optimization problems, stronger algorithmic, analytic, and geometric statements can be made when the functions involved in the description of the problem are convex compared to when they are not. Hence, it is natural to ask (as Shor did) whether one can test for convexity of functions in an efficient manner. Polynomial functions provide a convenient setting for a rigorous study of this question from a computational complexity viewpoint.





For a number of problems of applied nature, one is interested not in checking for convexity of a function globally, but only over a compact region. The most common region that arises in practical applications is a \emph{box} (i.e., a hyper-rectangle). For example, the variables in many real-world optimization problems are lower and upper bounded a priori, allowing the algorithm designer to restrict attention to a box. Similarly, in branch-and-bound approaches to nonlinear and nonconvex optimization, one often recursively breaks the original problem into a number of subproblems over smaller boxes and aims to provide upper and lower bounds on the minimum of the objective function over these smaller boxes. In such a scenario, the subproblems for which the objective function is convex over the associated box can often be solved to global optimality.

The problem of detecting convexity over a box, or the related problem of \emph{imposing} convexity over a box, appears in areas outside of optimization as well. In control theory for example, Lyapunov functions that certify properties of dynamical systems are sometimes required to be convex over the region of the space where the state variables lie~\cite{chesi2008establishing, ahmadi2018sos}. Hence, the need for algorithms that check convexity of a candidate Lyapunov function over a compact region---or those that find a Lyapunov function that meets the convexity requirement by construction---naturally arises. Similarly, in statistics, the well-known problem of \emph{convex regression}~\cite{xu2016faithful,hannah2013multivariate,lim2012consistency} is about finding a convex function (within a certain function class) that best agrees with observed data. In this situation again, imposing convexity of the function globally is often too restrictive. Indeed, one generally has access to a feature domain where the explanatory variables of the problem are known to lie. This domain, more often that not, is a box.



Motivated by these considerations, our goal in this paper is to study the complexity of testing convexity of a given function over a box (which can only be simpler than the problem of imposing convexity over a box). As is done in~\cite{NPhard_Convexity_MathProg}, we restrict our attention to polynomial functions since they appear routinely in applications and are amenable to complexity theoretic investigations due to their finite parametrization.  Our main result (Theorem \ref{th:np.hardness.convex}) shows that the problem of testing convexity of a polynomial over a box is strongly NP-hard already for cubic polynomials. This result completely classifies the complexity of testing convexity of polynomials of any degree over a box (cf. Proposition \ref{prop:degree}). It also justifies, at least from a complexity viewpoint, why convexity detection in some of the most widely-used optimization packages such as BARON~\cite{baron,khajavirad2017hybrid}, CVX~\cite{cvx}, Gurobi~\cite{gurobi}, \gh{and DR.AMPL/COCONUT~\cite{trees.conv.detection}} is restricted to quadratic functions, or is replaced with the task of checking \gh{some sufficient or necessary conditions for convexity that are more tractable. These sufficient conditions typically require} that the function in question be the output of certain convexity-preserving operations applied to an initial set of convex functions; see, e.g.,~\cite[Section 3.2]{BoydBook} \gh{or~\cite[Section 3]{trees.conv.detection}}. \gh{The necessary conditions, which can be used to disprove convexity, often involve a search for a direction of negative curvature at some sample points in the region of interest; see, e.g.,~\cite[Section 5]{trees.conv.detection} or \cite{mprobe}.} \gh{There are also algebraic sufficient conditions for convexity over a box that can be efficiently checked by sum of squares optimization; see, e.g., \cite[Section 5]{magnani2005tractable}, \cite[Chapter 7]{phdthesisGH}. These tests are mainly applicable to polynomial functions and to our knowledge are not currently implemented in popular software.}

\gh{Assuming P$\neq$ NP, the implication of our Theorem~\ref{th:np.hardness.convex} below for convexity detection is as follows. Suppose an optimization solver is endowed with a pseudo-polynomial time algorithm for proving convexity of a function over a box and a pseudo-polynomial time algorithm for disproving it. Then, regardless of how these algorithms are designed, there are cubic polynomials for which convexity can neither be proved nor disproved by this solver. In fact, as the number of variables grows, this solver will fail to prove or disprove convexity for a ``significant portion'' of cubic polynomials; see Corollary~2.2 of \cite{hemaspaandra2012sigact} for a precise complexity-theoretic statement.}




The proof of \gh{Theorem~\ref{th:np.hardness.convex}} has two parts, the first of which relies heavily on a clever gadget by Nemirovski~\cite{nemirovskii1993several}. Some of our modifications to his arguments are of potential interest to the field of robust control (cf. Corollary~\ref{cor:psd.interval}).

\section{The Main Result}

We begin with some basic definitions.

\begin{defn}
Given a set of scalars $l_1,\ldots,l_n$,$u_1,\ldots,u_n,$ with $l_i\leq u_i$ for $i=1,\ldots,n$, a \emph{box} $B\subseteq \mathbb{R}^n$ is a set of the form $$B=\{x \in \mathbb{R}^n|~ l_i\leq x_i \leq u_i, i=1,\ldots,n\}.$$
\end{defn}

\begin{defn}
A function $f:\mathbb{R}^n \rightarrow \mathbb{R}$ is said to be \emph{convex over a \gh{convex} set} $C\subseteq \mathbb{R}^n$ if for every $x, y \in C$ and for any $\lambda \in [0,1]$, we have $f(\lambda x+(1-\lambda)y)\leq \lambda f(x)+(1-\lambda)f(y)$.
\end{defn}

We study the complexity of detecting convexity in the standard Turing model of computation (see, e.g., \cite{sipser2006introduction}) where the input to every problem instance must be defined by a finite number of bits. As a consequence, in the statement of Theorem~\ref{th:np.hardness.convex} below, the input to the problem---which consists of the coefficients of the polynomial $f$ and the scalars $l_1,\ldots,l_n,u_1,\ldots,u_n$ that define the box $B$---is taken to be rational.  

We further remark that the problem in the statement of Theorem \ref{th:np.hardness.convex} is \emph{strongly} NP-hard. This is as opposed to problems that are \emph{weakly} NP-hard such as the classical problems of KNAPSACK or PARTITION~\cite{GareyJohnson_Book}. This distinction only occurs in problems whose instances involve numerical data\footnote{Note that the instances of not all of NP-hard problems involve numerical data; consider, e.g., SATISFIABILITY~\cite{GareyJohnson_Book}. For such problems, NP-hardness is by default in the strong sense.} and relates to how the numerical values that appear in an instance compare against the length of the instance. Roughly speaking, being NP-hard in the strong sense means that the problem remains NP-hard even when the numerical values in all problem instances are ``small''. More rigorously, \gh{following the conventions used in \cite{GareyJohnson_Book}}, define a Max (resp. Length) function which maps any instance $I$ of a decision problem to a nonnegative integer Max($I$) (resp. Length($I$)) that represents the largest numerator or denominator in magnitude of all rational numbers appearing in $I$ (resp. the number of bits required to write down $I$). A decision problem $\Pi$ is said to be \emph{strongly NP-hard} if there exists a subproblem $\hat{\Pi}$ which is NP-hard and for which one can find a polynomial $p$ such that Max($I$) is upperbounded by $p$(Length($I$)) for all instances $I \in \hat{\Pi}$. 
 
The advantage of showing that a problem is strongly NP-hard (as opposed to weakly NP-hard) is that such a statement rules out, unless P=NP, the possibility of a \emph{pseudo-polynomial time} algorithm. This is an algorithm whose running time is polynomial in the numerical value of the input but not necessarily in the bit length of the input; see~\cite[Section 4.2.2]{GareyJohnson_Book} for more details. More concretely, if we had only shown weak NP-hardness of testing convexity of an $n$-variate cubic polynomial over a box, it could be that the problem is polynomial-time solvable (even assuming P$\neq$ NP) except for instances where some of the coefficients of the polynomial or the bounds of the box are of order $2^n$. Indeed, for the KNAPSACK and PARTITION problems for example, it is well known that one can solve, in polynomial time, all instances which do not contain such large numbers using dynamic programming \cite[p. 96]{GareyJohnson_Book}.

%
%
%


\begin{theorem}[Main result] \label{th:np.hardness.convex}
	Given a cubic polynomial $f$ and a box $B$, it is strongly NP-hard to test whether $f$ is convex over $B$.
\end{theorem}
\begin{remark}
	\gh{In fact, testing convexity of a cubic polynomial $f$ over a box $B$ is a co-NP-complete problem. To see membership in co-NP, first observe that since the entires of $\nabla^2 f$ are affine, positive semidefiniteness of $\nabla^2 f$ at the extreme points of $B$ implies its positive semidefiniteness over all of $B$. Hence, if $f$ is not convex over $B$, one can always present one of the extreme points of $B$ as a certificate of nonconvexity. This certificate has polynomial size since its entries consist of the rational numbers $l_i,u_i$ which are part of the input. Moreover, once such a candidate extreme point is presented, one can conclude in polynomial time that the Hessian matrix evaluated at it is not positive semidefinite; see e.g. the arguments in the first paragraph of the proof of Proposition~\ref{prop:degree}.}
	\end{remark}

The proof of \gh{Theorem~\ref{th:np.hardness.convex}} is based on a reduction from the problem of testing whether a matrix $L(x)$---whose entries are affine polynomials in $x$---is positive semidefinite for all $x$ in a full-dimensional box $B$. This problem has already been shown to be \gh{(weakly)} NP-hard by Nemirovski~\cite{nemirovskii1993several}. The connection between this problem and the problem of testing convexity of a cubic polynomial $f\mathrel{\mathop{:}}=f(x)$ over a box $B$ lies in the facts that (i) the Hessian $\nabla^2 f(x)$ of $f(x)$ is an affine polynomial matrix, and (ii) a twice continuously differentiable function $f$ is convex over a convex set $S$ with nonempty interior if and only if \gh{$\nabla^2 f(x) \succeq 0,~\forall x \in S$; see, e.g., \cite[Chapter 1]{bertsekas2009convex}.\footnote{Here, the notation $M\succeq 0$ is used to denote that a symmetric matrix $M$ is positive semidefinite, i.e., has nonnegative eigenvalues. We will also use the notation $A \succeq B$ for two symmetric matrices $A$ and $B$ to denote that $A-B \succeq 0$.}} The proof of Theorem \ref{th:np.hardness.convex} is split into two parts which we outline below.


In the first part of the proof, we concern ourselves with strengthening the NP-hardness result from~\cite{nemirovskii1993several}, which we use here, from weak to strong (Theorem~\ref{th:strong.np.hard}). Indeed, the reduction provided in \cite{nemirovskii1993several} is from the PARTITION problem which is only weakly NP-hard~\cite{GareyJohnson_Book}, and hence only proves weak NP-hardness of testing positive semidefiniteness of an affine polynomial matrix over a box. 
%
As a consequence, we give a new reduction, which still uses the insights that we gleaned from Nemirovski's proof, but with two modifications. First, we give a reduction from a strongly NP-hard problem. Second, we bypass a step in the Nemirovski's construction that involves matrix inversion. This is because it is possible for the inverse of a matrix to have entries that are exponential in the entries of the original matrix, thus causing the reduction to lose its strong NP-hardness implication. To bypass matrix inversion, we instead approximate the inverse by its first-order Taylor expansion and control the spectrum of the matrix in such a way that the higher-order terms in the Taylor expansion can be appropriately bounded (cf. Lemma~\ref{lem:bounds}).




In the second and core part of the proof, we give the reduction from the problem of testing positive semidefiniteness of an affine polynomial matrix over a box to the problem of testing convexity of a cubic polynomial over a box. 
The main obstacle that we need to overcome here is that not every affine (symmetric) polynomial matrix $L(x)$ is the Hessian of some cubic polynomial. For instance, if, for some integers $i,j,k \in \{1,\ldots,n\}$, the equality $$\frac{\partial L_{ij}(x)}{\partial x_k} =\frac{\partial L_{ik}(x)}{\partial x_j}$$ is violated, then $L(x)$ cannot be a valid Hessian matrix. This is because partial derivatives of polynomials must commute. Our goal is then to show that even with these additional constraints, the problem of checking positive semidefiniteness remains hard. To do this, we will introduce new variables $y$ and carefully construct a polynomial $f(x,y)$ whose Hessian can be related back to $L(x)$.

\gh{Before we proceed, we briefly contrast the proof of the main result of this paper with the proof of the main result (Theorem 2.1) in \cite{NPhard_Convexity_MathProg} as they have some commonalities. Both proofs involve a reduction where it is shown that an $m\times m$ symmetric polynomial matrix $P(x)$ in $n$ variables is positive semidefinite over a region $R$ if and only if the Hessian matrix of some polynomial $f$ in a higher number of variables is positive semidefinite over a related region $R'$. In~\cite{NPhard_Convexity_MathProg}, $m=n$, the entries of $P(x)$ are homogeneous quadratic polynomials, $R$ and $R'$ are the entire Euclidean space, and the number of variables in $f$ is $2n$. In our proof, $m=n+1$, the entries of $P(x)$ are affine polynomials, $R$ and $R'$ are boxes, and the number of variables in $f$ is $2n+1$. In~\cite{NPhard_Convexity_MathProg}, the proof starts from the fact that testing positive semidefiniteness of a quadratic polynomial matrix on the unit sphere is strongly NP-hard. This fact is not helpful for us since quadratic polynomial Hessian matrices are second derivatives of quartic polynomials while our goal is to end up with cubic polynomials. Hence, we have to first establish strong NP-hardness of testing positive semidefiniteness of affine polynomial matrices (over boxes) as mentioned before. Another difference in the two proofs is in the construction of the polynomial $f$ which as compared to~\cite{NPhard_Convexity_MathProg} necessitates a more careful balancing of the coefficients of the terms involved. We also note that the reduction in~\cite{NPhard_Convexity_MathProg} turns \emph{any} homogeneous quadratic polynomial matrix into a valid Hessian matrix while preserving positive semidefiniteness. By contrast, our reduction exploits the specific structure of the affine polynomial matrix that emerges from Theorem~\ref{th:strong.np.hard} below; see $L(x)$ in (\ref{def:Lx}).
}

\begin{theorem}[see the proof of Proposition 2.1 in~\cite{nemirovskii1993several} for a related result]\label{th:strong.np.hard}
	Given a symmetric $n \times n$ matrix $A$ with entries in $\{0,1\}$ and a positive integer $k \leq n^2$, let
	\begin{align}\label{def:gam.C.mu}
	C\mathrel{\mathop{:}}=\frac{4}{(n+1)^3}\left(I_n+\frac{1}{(n+1)^3}A\right), \text{ and } \mu\mathrel{\mathop{:}}=\frac{n(n+1)^3}{4}+k-1-\frac{1}{4}e^TAe,
	\end{align}
	where $e$ is the $n \times 1$ vector of all ones and $I_n$ is the $n \times n$ identity matrix.
	It is strongly NP-hard to test whether 
	\begin{align}\label{def:Lx}
	L(x)\mathrel{\mathop{:}}=\begin{bmatrix}
	C & x \\ x^T & \mu +\frac14
	\end{bmatrix} \succeq 0, \forall x \in \mathbb{R}^n \text{ with } ||x||_{\infty}\leq 1.
	\end{align}
	Furthermore, for any matrix $C$ and scalar $\mu$ thus defined, either (\ref{def:Lx}) holds or there exists \gh{an} $x \in \mathbb{R}^n$ with $||x||_{\infty}\leq 1$ such that $x^TC^{-1}x \geq \mu +\frac34$.
\end{theorem}

\gh{
\begin{remark}\label{rem:Schur}
As the matrix $C$ is positive definite (see the proof of Lemma \ref{lem:bounds}), by using the Schur complement \cite[Appendix C.4]{BoydBook}, the constraint in (\ref{def:Lx}) is equivalent to having $x^TC^{-1}x \leq \mu+\frac14$  for all $x$ with $||x||_{\infty}\leq 1$. 
%
\end{remark}}

To prove \gh{Theorem~\ref{th:strong.np.hard}}, we will make use of the following lemma.

\begin{lemma}\label{lem:bounds}

Let $A$ be a symmetric $n \times n$ matrix with entries in $\{0,1\}$ and let $C$ be defined as in (\ref{def:gam.C.mu}). Then, for any $x \in \mathbb{R}^n$ with $||x||_{\infty} \leq 1$, we have
\begin{align*}
\frac{1}{4}x^T\left((n+1)^3I_n-A\right)x -\frac{1}{4}\leq x^TC^{-1}x \leq \frac{1}{4}x^T\left((n+1)^3I_n-A\right)x +\frac{1}{4}.
\end{align*}
\end{lemma}
\begin{proof}
First, note that $C$ is by construction strictly diagonally dominant, and hence, by Gershgorin's circle theorem \cite{gersh}, positive definite. This implies that $C^{-1}$ is well defined. Furthermore, we have
\begin{align}\label{eq:C.identity}
C^{-1}=\frac{(n+1)^3}{4} \left(I_n- \frac{-A}{(n+1)^3}\right)^{-1}=\frac{(n+1)^3}{4} \sum_{k=0}^{\infty} \frac{(-A)^k}{(n+1)^{3k}},
\end{align}
where the second equality follows from the identity $(I_n-M)^{-1}=\sum_{k=0}^\infty M^k$, \gh{which holds if for some matrix norm  $||\cdot||$ we have $||M||<1$; see, e.g., \cite[Section 5.6]{HJ_Matrix_Analysis_Book}. Here, for example, it is easy to see that the maximum absolute row sum norm of $\frac{-A}{(n+1)^3}$ is strictly less than one. Indeed, each row of $A$ contains at most $n$ ones, and so the maximum absolute row sum norm of $-A/(n+1)^3$ is less than $1/(n+1)^2<1$.}
%
%
%
 Using (\ref{eq:C.identity}), we have
$$x^TC^{-1}x=\frac{1}{4} x^T((n+1)^3I_n -A)x +\frac{(n+1)^3}{4} \sum_{k=2}^\infty \frac{x^T(-A)^kx}{(n+1)^{3k}}.$$
It remains to show that $$\left| \frac{(n+1)^3}{4} \sum_{k=2}^\infty \frac{x^T(-A)^kx}{(n+1)^{3k}} \right| \leq \frac14$$ for any $x \in \mathbb{R}^n$ with $||x||_{\infty} \leq 1$ to conclude the proof. Letting $x \in \mathbb{R}^n$ with $||x||_{\infty} \leq 1$, we have
\begin{align*}
\left| \frac{(n+1)^3}{4} \sum_{k=2}^\infty \frac{x^T(-A)^kx}{(n+1)^{3k}} \right| &\leq \frac{(n+1)^3}{4} \sum_{k=2}^\infty \frac{|x^T(-A)^kx|}{(n+1)^{3k}}\\
&\leq \frac{(n+1)^3}{4} \sum_{k=2}^\infty \frac{||x||_2\cdot ||(-A)^kx||_2}{(n+1)^{3k}}\\
&\leq \frac{(n+1)^3}{4} \sum_{k=2}^\infty \frac{||x||_2 \cdot ||A||^k_2 \cdot ||x||_2}{(n+1)^{3k}}\\
&\leq \frac{(n+1)^3}{4} \sum_{k=2}^\infty \frac{n||A||^k_2}{(n+1)^{3k}}\\
&\leq \frac{n(n+1)^3}{4} \sum_{k=2}^\infty \frac{1}{(n+1)^{2k}},
\end{align*}
where the first inequality uses the triangle inequality, the second the Cauchy-Schwarz inequality, the third the definition and the submultiplicativity property of the matrix 2-norm, the fourth the facts that $||x||_2 \leq \sqrt{n}||x||_{\infty}$ and $||x||_{\infty} \leq 1$, and the fifth the fact that $||A||_2\leq n \leq n+1.$ Now, using the geometric series identity 
$$\sum_{k=2}^{\infty} \frac{1}{(n+1)^{2k}}=\frac{1}{(n+1)^4}\cdot \frac{1}{1-\frac{1}{(n+1)^2}},$$ we conclude that
\begin{align*}
\left| \frac{(n+1)^3}{4} \sum_{k=2}^\infty \frac{x^T(-A)^kx}{(n+1)^{3k}} \right| &\leq \frac{1}{4}\cdot \frac{n}{n+1} \cdot\frac{1}{1-\frac{1}{(n+1)^2}}\\
&=\frac{1}{4}\cdot \frac{n+1}{n+2} \\
&\leq \frac{1}{4}.
\end{align*}
\end{proof}

\begin{proof}[Proof of Theorem \ref{th:strong.np.hard}]
	We give a reduction from the SIMPLE MAX-CUT problem, which is the following decision problem: Given a simple\footnote{Recall that a graph is \emph{simple} if it is unweighted, undirected, and has no self-loops or multiple edges.} graph $G=(V,E)$ and a positive integer $s \leq |V|^2$, decide whether there is a cut of size greater or equal to $s$ in $G$, i.e., a partition of $V$ into two disjoint sets $V_1$ and $V_2$ such that the number of edges from $E$ that have one endpoint in $V_1$ and one endpoint in $V_2$ is greater or equal to $s$. This problem is known to \gh{strongly} NP-hard~\cite[p. 210]{GareyJohnson_Book}.

Consider an instance of SIMPLE MAX-CUT given by a graph $G=(V,E)$ and a positive integer $s \leq |V|^2$. We construct an instance of the problem given in the statement of the theorem by taking $A$ to be the adjacency matrix of $G$ (i.e., a symmetric $|V| \times |V|$ matrix whose $(i,j)$-th entry is equal to one if $\{i,j\} \in E$ and equal to zero otherwise) and $k=s$. Clearly this reduction is polynomial in length. Now, let $C$ and $\mu$ be as defined in (\ref{def:gam.C.mu}) with $n=|V|$. 
	%

	
	We show that $L(x)$ as defined in (\ref{def:Lx}) is positive semidefinite for all $x \in \mathbb{R}^n$ with $||x||_{\infty} \leq 1$ if and only if there is no cut in $G$ of size greater or equal to $s$. This would show that the problem in the statement of the theorem is NP-hard. In fact, as explained previously, since $$\mbox{Max}(I)=k \leq n^2+\lceil \log_2(k) \rceil=\mbox{Length}(I)$$ for any instance $I$ of the problem, the problem is automatically strongly NP-hard; see e.g. \cite[p. 95]{GareyJohnson_Book}. 
	
	Suppose first that there is no cut in $G$ of size greater or equal to $s$. As $k=s$, this implies that the largest cut in $G$ is of size less than or equal to $k-1$. It is straightforward to verify that the size of the largest cut in $G$ is given by $\max_{x \in \{-1,1\}^n} \frac{1}{4}\sum_{i,j}A_{ij} (1-x_ix_j)$. Hence, if 
	\begin{equation*}
	\begin{aligned}
	p^*\mathrel{\mathop{:}}=&\min_x &&\frac{1}{4}x^TAx\\
	&\text{s.t. } &&x \in \{-1,1\}^n,
	\end{aligned}
	\end{equation*}
	the size of the largest cut in $G$ is less than or equal to $k-1$ if and only if  $\frac{1}{4}e^TAe -p^* \leq k-1 \Leftrightarrow p^{*} \geq \frac{1}{4} e^TAe -k+1.$ As $x^Tx=n$ when $x \in \{-1,1\}^n$, this is equivalent (after basic algebra) to the optimal value of 
	\begin{equation*}
	\begin{aligned}
	&\max_x &&\frac{1}{4} x^T((n+1)^3 I_n-A)x\\
	&\text{s.t. } &&x \in \{-1,1\}^n
	\end{aligned}
	\end{equation*}
	being less than or equal to $\frac{n(n+1)^3}{4} +k-1-\frac{1}{4}e^TAe=\mu.$ As the quadratic function $$x \mapsto \frac14 x^T((n+1)^3I_n-A)x$$ is convex, and as the maximum of a convex function over a box is attained at an extreme point of the box (see, e.g., \cite[Property 12]{bensonconcave}), the previous statement is equivalent to the optimal value of
	\begin{equation*}
	\begin{aligned}
	&\max_x &&\frac{1}{4} x^T((n+1)^3I_n-A)x\\
	&\text{s.t. } &&x \in [-1,1]^n
	\end{aligned}
	\end{equation*}
	being less than or equal to $\mu.$ Hence we have shown that if the size of the largest cut in $G$ is less than or equal to $k-1$, then $\frac{1}{4} x^T ((n+1)^3 I_n-Ax)x \leq \mu$ for all $x \in \mathbb{R}^n$ with $||x||_{\infty} \leq 1.$ From Lemma \ref{lem:bounds}, it follows that 
	$$x^TC^{-1}x \leq \mu +\frac{1}{4}$$ for any $x \in \mathbb{R}^n$ with $||x||_{\infty}\leq 1.$ Using the Schur complement \gh{(see Remark \ref{rem:Schur})}, this is equivalent to $L(x)$ being positive semidefinite for all $x \in \mathbb{R}^n$ with $||x||_{\infty}\leq 1.$
	
	Suppose now that there is a cut in $G$ of size greater or equal to $s$. Let $\hat{x} \in \{-1,1\}^n$ be the indicator vector of this cut; i.e., $\hat{x}_i=1$ if node $i$ belongs to $V_1$ and $\hat{x}_i=-1$ if node $i$ belongs to $V_2$. It is easy to check that the number of edges with one endpoint in $V_1$ and one endpoint in $V_2$ is given by $\frac{1}{4} (e^TAe- \hat{x}^TA\hat{x}).$  Hence, as $k=s$, we have
	$$\frac{1}{4} (e^TAe- \hat{x}^TA\hat{x}) \geq k,$$ which is equivalent to $\frac{1}{4} n(n+1)^3-\frac{1}{4}\hat{x}^TA\hat{x} \geq \frac{1}{4} n(n+1)^3-\frac{1}{4}e^TAe+k=\mu+1$. As $\hat{x}^T\hat{x}=n$, it follows that 
	$$\frac{1}{4}\hat{x}^T((n+1)^3I_n-A)\hat{x} \geq \mu +1.$$
	Using Lemma \ref{lem:bounds}, we get $$\hat{x}^TC^{-1}\hat{x} \geq \mu +\frac34 >\mu +\frac14.$$
By the Schur complement, we conclude that $L(\hat{x}) \not\succeq 0.$
	
	Note that the final claim of the theorem is already implied by the arguments given thus far. Indeed, given $C$ and $\mu$ as defined in (\ref{def:gam.C.mu}), one can construct a graph whose adjacency matrix is $A$ and take $s=k$. Either this graph has no cut of size greater or equal to $s$, in which case we have shown that $x^TC^{-1}x \leq \mu+\frac14$ for all $x \in \mathbb{R}^n$ with $||x||_{\infty} \leq 1$, or it has a cut of size greater or equal to $s$, in which case there exists $\hat{x}$ such that $||\hat{x}||_{\infty} \leq 1$ and $\hat{x}^T C^{-1} \hat{x} \geq \mu +\frac{3}{4}.$
\end{proof}

We briefly provide an immediate corollary of this theorem, which we believe can be of independent interest as it relates to problems in robust control~\cite{ben2003extended,blondel1997np,ben2002tractable}. This statement was first proven in \cite{nemirovskii1993several}, except that the result there shows NP-hardness in the weak sense. 
%


\begin{corollary}[Strong NP-hardness of testing interval positive semidefiniteness]\label{cor:psd.interval}
Given rational numbers $\hat{m}_{ij}, \bar{m}_{ij}, i=1,\ldots,N, j=1,\ldots,N,$ with $\hat{m}_{ij} \leq \bar{m}_{ij}$, $\hat{m}_{ij}=\hat{m}_{ji}$, and $\bar{m}_{ij}=\bar{m}_{ji}$ for $i=1,\ldots,N,$ $j=1,\ldots,N$, it is strongly NP-hard to test whether all $N \times N$ symmetric matrices $M$ with entries $m_{ij} \in [\hat{m}_{ij}, \bar{m}_{ij}]$ for $i=1,\ldots,N, j=1,\ldots,N$, are positive semidefinite. 
%
%
\end{corollary}
\begin{proof}
The claim follows straightforwardly by a reduction from the problem in the statement of Theorem \ref{th:strong.np.hard}. Let $A$ be a symmetric $n \times n$ matrix with entries in $\{0,1\}$ and $k \leq n^2$ be a positive integer. Construct the symmetric matrix $C$ and the scalar $\mu$ as in (\ref{def:gam.C.mu}). Let $\hat{m}_{ij}=\bar{m}_{ij}=C_{ij}$ for $i=1,\ldots,n, j=1,\ldots,n,$ $\hat{m}_{i(n+1)}=\hat{m}_{(n+1)i}=-1$ and $\bar{m}_{i(n+1)}=\bar{m}_{(n+1)i}=1$ for $i=1,\ldots,n$, and $\hat{m}_{(n+1)(n+1)}=\bar{m}_{(n+1)(n+1)}=\mu+\frac{1}{4}$. It follows that all $(n+1) \times (n+1)$ symmetric matrices $M$ with entries $m_{ij} \in [\hat{m}_{ij}, \bar{m}_{ij}]$ for $i=1,\ldots,n+1, j=1,\ldots,n+1$, are positive semidefinite if and only if the matrix $L(x)$ defined in (\ref{def:Lx}) is positive semidefinite for all $x \in \mathbb{R}^n$ with $||x||_{\infty} \leq 1.$ 

Note that the reduction we have given leads to a strong NP-hardness result. Indeed, all instances $I$ of the interval positive semidefiniteness problem that arise from this construction have $\mbox{Length}(I)$ of order $n^2$ and $\mbox{Max}(I)$ of order $n^6$, and hence $\mbox{Max}(I)$ is upper bounded by a polynomial in $\mbox{Length}(I)$.
%
%
%
%
\end{proof}

We now move on to the proof of the main theorem of this paper.

\begin{proof}[Proof of Theorem \ref{th:np.hardness.convex}]
	
	We prove the claim via a reduction from the NP-hard problem stated in Theorem \ref{th:strong.np.hard}. Let $A$ be an $n \times n$ symmetric matrix with entries in $\{0,1\}$ and $k \leq n^2$ be a positive integer. Let $C$ and $\mu$ be as in (\ref{def:gam.C.mu}) and, \gh{for $x\in \mathbb{R}^n$,} define $L(x)$ to be the following $(n+1) \times (n+1)$ symmetric matrix  $$L(x) \mathrel{\mathop{:}}=\begin{bmatrix} C & x\\ x^T & \mu +\frac14 \end{bmatrix}.$$ Let $y\mathrel{\mathop{:}}=(y_1,\ldots,y_{n+1})^T$ be a new vector of variables and  $H(y)$ be the following $n \times (n+1)$ matrix of mixed partial derivatives of the cubic polynomial $y^TL(x)y$: $$H_{ij}(y)=\frac{\partial^2 y^TL(x)y}{\partial x_i \partial y_j},~\text{for } i=1,\ldots,n, \text{ and }j=1,\ldots,n+1.$$ 
	Using the structure of $L(x)$, after some basic algebra, we obtain:
	\begin{align}\label{eq:Hy}
	H(y)=\begin{bmatrix} 2y_{n+1} & 0 & \ldots & 0 & 2y_1 \\ 0 & \ddots & 0 & 0  & 2y_2\\ 0 & 0 & \ddots & 0 & \vdots \\  0 & \ldots & 0 & 2y_{n+1} & 2y_n \end{bmatrix}.
	\end{align}
%
%
%
%
%
%
%
Let
	\begin{equation}\label{def:n.alpha.gamma}
	\begin{aligned} 
	&\alpha\mathrel{\mathop{:}}=16\cdot n(1+ 16n^7),\\
	\gh{\text{and }}&\eta\mathrel{\mathop{:}}=\frac{1}{4}\cdot \frac{1}{1+16n^7}.
	\end{aligned} 
	\end{equation} Consider the following cubic polynomial in $2n+1$ variables $x_1,\ldots,x_n, y_1,\ldots,y_{n+1}$ $$f(x,y)=\frac{1}{2}y^TL(x)y+\frac{\alpha}{2} x^Tx+\frac{\eta}{2} y^Ty,$$ and the box $B=[-1,1]^{2n+1}.$
	We claim that $f(x,y)$ is convex over $B$ if and only if $L(x)\succeq 0$ for all $x \in \mathbb{R}^n$ with $||x||_{\infty} \leq 1.$ This would imply the desired NP-hardness result as the construction of $f$ and $B$ from $A$ and $k$ can be carried out in polynomial time. Moreover, NP-hardness here is in the strong sense as all instances $I$ of the problem in the statement of Theorem \ref{th:np.hardness.convex} that arise from this construction have $\mbox{Length}(I)$ of order $n^3$ (the number of coefficients of $f$) and $\mbox{Max}(I)$ of order $n^8$ (which is the magnitude of $\alpha$, which dominates all other coefficients). Hence, $\mbox{Max}(I)$ is upper bounded by a polynomial in \gh{$\mbox{Length}(I).$} 	
%
%

To prove the claim, we begin by observing that the Hessian of $f$ has the following structure:
$$\nabla^2 f(x,y)=\begin{bmatrix}\alpha I_n  & \frac12 H(y) \\ \frac12 H(y)^T & L(x)+\eta I_{n+1}\end{bmatrix}.$$ 
%
%

%
%
	
Suppose first that $L(x)$ is not positive semidefinite for all $x\in [-1,1]^n$. Then, from the second claim in Theorem \ref{th:strong.np.hard}, there exists $\bar{x} \in [-1,1]^n$ such that $\bar{x}^TC^{-1}\bar{x} \geq \mu+\frac34.$ Let $\bar{y}=0_{(n+1)\times 1}$, i.e., the vector of zeros of length $n+1$, and observe that $(\bar{x},\bar{y}) \in [-1,1]^{2n+1}=B.$ Let $$z=\frac{1}{||(-C^{-1}\bar{x},1)^T||_2}\begin{pmatrix} 0_{n \times 1}\\ -C^{-1}\bar{x}\\ 1
	\end{pmatrix}.$$ 
	We will show that $z^T\nabla^2 f(\bar{x},\bar{y})z < 0.$ Indeed, we have 
	\begin{equation}\label{eq:expand.hessian}
	\begin{aligned}
	z^T\nabla^2f(\bar{x},\bar{y})z&= \frac{1}{1+||C^{-1}\bar{x}||_2^2} \cdot\begin{bmatrix} -C^{-1}\bar{x}\\ 1 \end{bmatrix}^T (L(\bar{x})+\eta I_{n+1})\begin{bmatrix} -C^{-1}\bar{x}\\ 1 \end{bmatrix} \\
	&= \frac{1}{1+||C^{-1}\bar{x}||_2^2} \cdot\begin{bmatrix} -C^{-1}\bar{x}\\ 1 \end{bmatrix}^T \begin{bmatrix}
	C+\eta I_n & \bar{x} \\ \bar{x}^T & \mu+\frac14 +\eta
	\end{bmatrix}\begin{bmatrix} -C^{-1}\bar{x}\\ 1 \end{bmatrix}\\
	&= \frac{1}{1+||C^{-1}\bar{x}||_2^2} \left( \mu +\frac14 -\bar{x}^TC^{-1}\bar{x} +\eta (1+||C^{-1}\bar{x}||_2^2) \right)\\
	&=\frac{\mu+\frac14-\bar{x}^TC^{-1}\bar{x}}{1+||C^{-1}\bar{x}||_2^2}+\eta.
	\end{aligned}
	\end{equation}
To show that this expression is negative, we upper bound the numerator and the denominator of its first summand. Using (\ref{eq:C.identity}), we can write 
	\begin{align*}
	||C^{-1}\bar{x}||_2 &\leq \frac{(n+1)^3}{4} \sum_{k=0}^{\infty} \frac{||A^k\bar{x}||_2}{(n+1)^{3k}}\\
	&\leq \frac{(n+1)^3}{4} \sum_{k=0}^{\infty} \frac{||A||_2^k\cdot || \bar{x}||_2}{(n+1)^{3k}}\\
	&\leq \frac{\sqrt{n}(n+1)^3}{4} \sum_{k=0}^{\infty} \frac{1}{(n+1)^{2k}}\\
	&\leq \frac{(n+1)^4}{4\sqrt{n}}
	\end{align*}
	where, much as was done in the proof of Lemma \ref{lem:bounds}, the first inequality uses the triangle inequality, the second the definition and the submultiplicativity property of the matrix 2-norm, the third the facts that $||\bar{x}||_2 \leq \sqrt{n}||\bar{x}||_\infty \leq \sqrt{n} $ and $||A||_2 \leq n+1$, and the last the formula for the sum of a geometric series and the fact that $\frac{n+1}{n+2} \leq 1$. 
		Hence: 
		\begin{align*}
		1+||C^{-1}\bar{x}||_2^2&{\gh \leq}1+ \frac{(n+1)^8}{16n} \leq 1+\frac{(2n)^8}{16n} {\gh =}1+16n^7.
		\end{align*}
	Combining this with (\ref{eq:expand.hessian}) and the fact that
		\begin{align*}
 \mu+\frac14 -\bar{x}^T C^{-1}\bar{x} \leq -\frac12,
		\end{align*}
	we get $$z^T\nabla^2 f(\bar{x},\bar{y})z\leq -\frac{1}{2(1+16n^7)}+\eta.$$
	Replacing $\eta$ by its expression in (\ref{def:n.alpha.gamma}), we see that $z^T \nabla^2 f(\bar{x},\bar{y})z \leq -\frac{1}{4} \cdot \frac{1}{1+16n^7}<0.$ Hence $\nabla^2 f(x,y)$ is not positive semidefinite over $B$, and therefore $f(x,y)$ is not convex over $B$.
	
	Suppose now that $L(x) \succeq 0$ for all $x \in [-1,1]^n.$ We prove that $f(x,y)$ is convex over $B$ by showing that $\nabla^2 f(x,y) \succeq 0$ for all $(x,y) \in B.$ As $\alpha>0$, this is equivalent to showing, using the Schur complement \gh{(see Remark \ref{rem:Schur})}, that $$L(x)+\eta I_{n+1} -\frac{1}{4\alpha} H(y)^TH(y) \succeq 0, \text{ for all } (x,y) \in B.$$
	We prove that $\eta I_{n+1}-\frac{1}{\alpha} H(y)^TH(y) \succeq 0$ for all $y \in [-1,1]^{n+1}$, and as $L(x) \succeq 0$ for any $x \in [-1,1]^n$, the claim would follow. From (\ref{eq:Hy}), we deduce the following expression for the symmetric $(n+1) \times (n+1)$ matrix $H(y)^TH(y)$:
	$$H(y)^TH(y)=4\begin{bmatrix} y_{n+1}^2 I_n & y_{n+1} \cdot y \\ y_{n+1} \cdot y^T & \sum_{i=1}^n y_i^2 \end{bmatrix}.$$
	By Gershgorin's circle theorem \cite{gersh}, for \gh{any} $y \in [-1,1]^{n+1}$, we can upper bound the largest eigenvalue of this matrix as follows:	\gh{
	\begin{align*}
	\lambda_{\max}(H(y)^TH(y)) &\leq 4 \max \left\{\max_{i=1,\ldots,n} \{y_{n+1}^2+|y_iy_{n+1}|\}, \sum_{i=1}^n y_i^2+\sum_{i=1}^n |y_{n+1}y_i|\right\}\\
	&\leq 4 \max \{2,2n\}=8n.
	\end{align*}}
Using this bound and replacing $\alpha$ and $\eta$ by their expressions in (\ref{def:n.alpha.gamma}), we conclude that 
	$$\eta I_{n+1}-\frac{1}{4\alpha} H(y)^TH(y) \succeq \left(\eta-\frac{2n}{\alpha}\right)I_{n+1} = \frac{1}{8(1+16n^7)}I_{n+1} \succeq 0.$$	
\end{proof}

\gh{
\begin{remark}
	As observed by a referee, a careful analysis of the proof of Theorem~\ref{th:np.hardness.convex} shows that it is also strongly NP-hard to check convexity of a cubic polynomial over an open hyper-rectangle (i.e., a set of the type $\{x \in \mathbb{R}^n ~|~ l_i <x_i<u_i,i=1,\ldots,n\}$).
\end{remark}
}

We end with the observation that Theorem \ref{th:np.hardness.convex} immediately classifies the complexity of detecting convexity of a polynomial of degree $d$ over a box, for any integer $d \geq 1$.

\begin{proposition}\label{prop:degree}
	The problem of testing whether a polynomial $f$ of degree $d$ is convex over a box $B$ can be solved in polynomial time for $d=1$ and $d=2$ and is strongly NP-hard for any fixed integer $d\geq 3.$
	\end{proposition}

\begin{proof}
When $d=1$, $f$ is always globally convex. Hence the problem of testing whether $f$ is convex over a box $B$ can trivially be solved in polynomial time (by answering `yes' to all instances). When $d=2$, the Hessian of $f$ is constant, and so testing convexity of $f$ over $B$ is equivalent to testing whether its (constant) Hessian matrix $\nabla^2 f$ is positive semidefinite. This can be done in polynomial time, e.g. by performing Gaussian pivot steps along the main diagonal of $\nabla^2 f$~\cite{nonnegativity_NP_hard}, or by computing the characteristic polynomial of $\nabla^2 f$ exactly and then checking that the signs of its coefficients alternate~\cite[p. 403]{HJ_Matrix_Analysis_Book}.


For $d=3$, Theorem \ref{th:np.hardness.convex} establishes the claim. For $d\geq 4$, we give a reduction from the problem of testing convexity of a cubic polynomial $g(x)\mathrel{\mathop{:}}=g(x_1,\ldots,x_n)$ over a full-dimensional\footnote{Note that the proof of Theorem \ref{th:np.hardness.convex} established strong NP-hardness of this problem as the box that arose in the proof was the unit hypercube, which is full-dimensional.} box $\tilde{B} \subseteq \mathbb{R}^n$. Given such $g, \tilde{B},$ let
$$f(x_1,\ldots,x_n,x_{n+1})\mathrel{\mathop{:}}=g(x_1,\ldots,x_n)+x_{n+1}^d,$$
and $B\mathrel{\mathop{:}}=\tilde{B}\times [0,1]$. We have
$$\nabla^2 f(x,x_{n+1})=\begin{bmatrix} \nabla^2 g(x) & 0\\ 0 & d(d-1)\cdot x_{n+1}^{d-2}\end{bmatrix},$$
and hence $\nabla^2 f(x,x_{n+1})\succeq 0$ over $B$ if and only if $\nabla^2 g(x)\succeq 0$ over $\tilde{B}$ (as $d(d-1)x_{n+1}^{d-2}$ is nonnegative over $[0,1]$). It follows that testing convexity of $f$ over $B$ is strongly NP-hard. 
%
%
%
%
\end{proof}

	
To conclude, we would like to emphasize that our result should not discourage researchers from seeking algorithms for testing convexity of polynomials over a box. \gh{While we have shown that} a search for a pseudo-polynomial time algorithm that works on all cubic polynomials is hopeless unless P=NP, \gh{it remains valuable to characterize interesting classes of polynomials on which convexity detection can be done efficiently.}


%
%
%
%
%
%
%
%
\section*{Acknowledgements}
We are grateful to two anonymous referees whose detailed and constructive feedback has improved this paper significantly.

\small{
\bibliographystyle{spmpsci}
\bibliography{thesis}}




\end{document}